\documentclass[12pt,dvips]{amsart}
\usepackage{amsfonts, amssymb, latexsym, epsfig}
\usepackage{euler, amsfonts, amssymb, latexsym, epsfig, epic}
\newtheorem{Theorem}{Theorem}[section]
\newtheorem{Claim}[Theorem]{Claim}
\newtheorem{Proposition}[Theorem]{Proposition}
\newtheorem{Lemma}[Theorem]{Lemma}

\newtheorem{Corollary}[Theorem]{Corollary}

\newtheorem{Definition-Proposition}[Theorem]{Definition-Theorem}
\newtheorem{Main Conjecture}[Theorem]{Main Conjecture}

\theoremstyle{remark}
\newtheorem{Example}[Theorem]{Example}

\newcommand\bor{{\mathfrak B}}

\theoremstyle{plain}

\newtheorem{Question}{Question}

\newcommand{\comment}[1]{$\star${\sf\textbf{#1}}$\star$}

\newcommand{\cellsize}{15}
\newlength{\cellsz} 
\newsavebox{\cell}
\sbox{\cell}{\begin{picture}(\cellsize,\cellsize)
\put(0,0){\line(1,0){\cellsize}}
\put(0,0){\line(0,1){\cellsize}}
\put(\cellsize,0){\line(0,1){\cellsize}}
\put(0,\cellsize){\line(1,0){\cellsize}}
\end{picture}}
\newcommand\cellify[1]{\def\thearg{#1}\def\nothing{}%
\ifx\thearg\nothing
\vrule width0pt height\cellsz depth0pt\else
\hbox to 0pt{\usebox{\cell} \hss}\fi%
\vbox to \cellsz{
\vss
\hbox to \cellsz{\hss$#1$\hss}
\vss}}
\newcommand\tableau[1]{\vtop{\let\\\cr
\baselineskip -16000pt \lineskiplimit 16000pt \lineskip 0pt
\ialign{&\cellify{##}\cr#1\crcr}}}
\newcommand{\kellsize}{17}
\newlength{\kellsz} 
\newsavebox{\kell}
\sbox{\kell}{\begin{picture}(\kellsize,\kellsize)
\put(0,0){\line(1,0){\kellsize}}
\put(0,0){\line(0,1){\kellsize}}
\put(\kellsize,0){\line(0,1){\kellsize}}
\put(0,\kellsize){\line(1,0){\kellsize}}
\end{picture}}
\newcommand\kellify[1]{\def\thearg{#1}\def\nothing{}%
\ifx\thearg\nothing
\vrule width0pt height\kellsz depth0pt\else
\hbox to 0pt{\usebox{\kell} \hss}\fi%
\vbox to \kellsz{
\vss
\hbox to \kellsz{\hss$#1$\hss}
\vss}}
\newcommand\ktableau[1]{\vtop{\let\\\cr
\baselineskip -16000pt \lineskiplimit 16000pt \lineskip 0pt
\ialign{&\kellify{##}\cr#1\crcr}}}

\newcommand{\sellsize}{36}
\newlength{\sellsz} 
\newsavebox{\sell}
\sbox{\sell}{\begin{picture}(\sellsize,20)
\put(0,0){\line(1,0){\sellsize}}
\put(0,0){\line(0,1){\sellsize}}
\put(\sellsize,0){\line(0,1){\sellsize}}
\put(0,\sellsize){\line(1,0){\sellsize}}
\end{picture}}
\newcommand\sellify[1]{\def\thearg{#1}\def\nothing{}%
\ifx\thearg\nothing
\vrule width0pt height\sellsz depth0pt\else
\hbox to 0pt{\usebox{\sell} \hss}\fi%
\vbox to \sellsz{
\vss
\hbox to \sellsz{\hss$#1$\hss}
\vss}}
\newcommand\stableau[1]{\vtop{\let\\\cr
\baselineskip -16000pt \lineskiplimit 16000pt \lineskip 0pt
\ialign{&\sellify{##}\cr#1\crcr}}}

\newcommand{\smellsize}{7}
\newlength{\smellsz} 
\newsavebox{\smell}
\sbox{\smell}{\begin{picture}(\smellsize,\smellsize)
\put(0,0){\line(1,0){\smellsize}}
\put(0,0){\line(0,1){\smellsize}}
\put(\smellsize,0){\line(0,1){\smellsize}}
\put(0,\smellsize){\line(1,0){\smellsize}}
\end{picture}}
\newcommand\smellify[1]{\def\thearg{#1}\def\nothing{}%
\ifx\thearg\nothing
\vrule width0pt height\smellsz depth0pt\else
\hbox to 0pt{\usebox{\smell} \hss}\fi%
\vbox to \smellsz{
\vss
\hbox to \smellsz{\hss$#1$\hss}
\vss}}
\newcommand\smtableau[1]{\vtop{\let\\\cr
\baselineskip -16000pt \lineskiplimit 16000pt \lineskip 0pt
\ialign{&\smellify{##}\cr#1\crcr}}}

\begin{document}
\pagestyle{plain}

\mbox{}
\title{The direct sum map on Grassmannians and\\ jeu de taquin for increasing tableaux}
\author{Hugh Thomas}
\address{Department of Mathematics and Statistics, University of New
Brunswick, Fredericton, New Brunswick, E3B 5A3, Canada }
\email{hugh@math.unb.ca}

\author{Alexander Yong}
\address{Department of Mathematics, University of Illinois at
Urbana-Champaign, Urbana, IL 61801, USA}
\email{ayong@math.uiuc.edu}

\date{January 31, 2010}

\begin{abstract}
The direct sum map $Gr(a,{\mathbb C}^n) \times Gr(b,{\mathbb C}^m)\rightarrow Gr(a+b,{\mathbb C}^{m+n})$
on Grassmannians induces a $K$-theory pullback that defines the  \emph{splitting coefficients}.
We geometrically explain an identity from [Buch '02]
between the splitting coefficients
and the Schubert structure constants for products of Schubert structure sheaves.
This is related to the topic of product and splitting coefficients for Schubert boundary ideal sheaves.
Our main results extend
\emph{jeu de taquin for increasing tableaux} [Thomas-Yong '09] by proving transparent
analogues of [Sch\"{u}tzenberger '77]'s fundamental theorems on well-definedness of rectification.
We then establish that \emph{jeu de taquin} gives rules for each of these four kinds of coefficients.
\end{abstract}
\maketitle


\section{Introduction}

This paper studies
\emph{jeu de taquin for increasing tableaux}, introduced in
application
to $K$-theoretic Schubert calculus \cite{Thomas.Yong:V}; see also \cite{Thomas.Yong:VI, Thomas.Yong:X}.
We establish new analogues of
M.-P.~Sch\"{u}tzenberger's fundamental theorems of classical {\it jeu de taquin}
\cite{Schutzenberger}. We then obtain new {\it jeu de taquin} rules, for splitting
coefficients on the
$K$-theory of Grassmannians, and for splitting and product coefficients for Schubert boundary ideal sheaves.

\subsection{Four kinds of coefficients in the $K$-theory of Grassmannians}
Let $X=Gr(k,{\mathbb C}^n)$ denote the Grassmannian of $k$-dimensional planes
in
${\mathbb C}^n$. This text further studies the algorithms of \cite{Thomas.Yong:V}.
These were used there to give a combinatorial rule (see Theorem~\ref{thm:TY5main} below) for the {\bf Schubert
structure constants} $C_{\lambda,\mu}^{\nu}\in {\mathbb Z}$
in the Grothendieck ring $K^{0}(X)$ of algebraic vector bundles on $X$,
with respect to multiplication in
the basis of Schubert structure sheaves $\{[{\mathcal O}_{X_{\lambda}}]\}$. That is to say, $C^\nu_{\lambda\mu}$ is defined by
$[{\mathcal O}_{X_{\lambda}}]\cdot [{\mathcal O}_{X_{\mu}}]
=\sum_{\nu} C_{\lambda,\mu}^{\nu} [{\mathcal O}_{X_{\nu}}]$, where
$\lambda,\mu,\nu$ are Young diagrams contained inside the ambient rectangle
$\Lambda:=k\times (n-k)$.
A different combinatorial rule for calculating the
coefficients $C_{\lambda,\mu}^{\nu}$ was earlier established by A.~Buch~\cite{Buch:KLR};
other rules have since followed (see \cite{Thomas.Yong:V} and the references therein).

When $|\nu|=|\lambda|+|\mu|$ (where $|\lambda|=\lambda_1+\lambda_2+\cdots+\lambda_k$ etc.), $C_{\lambda,\mu}^{\nu}$ are the
{\bf classical Littlewood-Richardson coefficients} $c_{\lambda,\mu}^{\nu}$
governing $H^{\star}(X)$,
the ordinary cohomology ring of $X$. These numbers can be viewed as
counting points in
the intersection of generic $GL_n$-translates of $X_{\lambda},X_{\mu}$ and
$X_{\nu^{\vee}}$ (where $\nu^{\vee}$ is the $180$ degree rotation of the
complement
of $\nu$ in $\Lambda$).

The combinatorics supporting our rule for $C_{\lambda,\mu}^{\nu}$ extended the \emph{jeu de taquin} theory introduced by
M.-P.~Sch\"{u}tzenberger \cite{Schutzenberger}. The linchpin of his approach consists of two deep results, sometimes
called the \emph{first and second fundamental theorems of jeu de taquin}.
These key theorems of the classical theory do not
\emph{naively} extend to our $K$-theory setting. Small counterexamples are shown
in \cite[Section~8]{Thomas.Yong:V} (see also Section~6),
indicating combinatorial
obstructions. However, in \cite{Thomas.Yong:V}, we proved a
form of these fundamental theorems, sufficient to give a rule to compute $C_{\lambda,\mu}^{\nu}$.
That result notwithstanding, we have since desired conceptual explanations
for these obstructions, both for intrinsic interest, but also, e.g., to have a
better
foundation to understand our extensions
of {\it jeu de taquin} to equivariant cohomology
and equivariant $K$-theory of $X$ \cite{Thomas.Yong:VII}.

Towards this end, one can geometrically motivate associating a
different collection of numbers to
the triple $(\lambda,\mu,\nu)$. Suppose $k_1<n_1 \mbox{\ and \ } k_2<n_2$
are two pairs of positive integers and $k=k_1+k_2, \ n=n_1+n_2$.
Let
\[Y=Gr(k_1, {\mathbb C}^{n_1}) \mbox{ \ and \ } Z=Gr(k_2,{\mathbb C}^{n_2}).\]
There is the {\bf direct sum map}
\begin{eqnarray}\nonumber
\psi: & Y\times Z & \to X\\ \nonumber
 &      (V_1,V_2) &  \mapsto V_1\oplus V_2 \ \ \ \ \ \subseteq
{\mathbb C}^{n_1}\oplus {\mathbb C}^{n_2}\cong
{\mathbb C}^n.\nonumber
\end{eqnarray}
Consider the pullback $\psi^{\star}$ on both
cohomology and $K$-theory. In the former case, we have
\[\psi^{\star}:H^{\star}(X)\to H^{\star}(Y\times Z)\cong
H^{\star}(Y)\otimes
H^{\star}(Z).\]
The latter case replaces ``$H^{\star}$'' by
``$K^{0}$'' throughout.
These pullbacks can be defined explicitly on the Schubert basis.

In cohomology,
the classical Littlewood-Richardson coefficients
also define the pullback:
\begin{equation}
\label{eqn:conflate}
\psi^{\star}([X_{\nu}])=\sum_{\nu} c_{\lambda,\mu}^{\nu}\
[Y_\lambda]\otimes
[Z_{\mu}],
\end{equation}
where $[X_{\nu}], [Y_{\lambda}]$ and $[Z_{\mu}]$ are Schubert classes in
$H^{\star}(X), H^{\star}(Y)$ and $H^{\star}(Z)$ respectively. This notation is explained in Section~2.2.

Similarly, the $K$-theory pullback can be described as
\[\psi^{\star}([{\mathcal O}_{X_{\nu}}])=
\sum_{\nu} D_{\lambda,\mu}^{\nu}\ [{\mathcal O}_{Y_{\lambda}}]\otimes
[{\mathcal O}_{Z_{\mu}}],\]
but, in contrast, the {\bf splitting coefficient} $D_{\lambda,\mu}^{\nu}$ is not generally equal to the Schubert structure
constant $C_{\lambda,\mu}^{\nu}$.

The numbers
$D_{\lambda,\mu}^{\nu}$ were first treated, along with $C_{\lambda,\mu}^{\nu}$,
in \cite{Buch:KLR}, who gave a simple identity relating the two collections
of numbers.
We give a geometric proof of an essentially equivalent identity in Section~2, by
seeing the image of a product of Schubert varieties under $\psi$
as a certain Richardson variety inside $X$. This relation is used in our proof
of a new combinatorial rule for $D_{\lambda\mu}^\nu$
in Section~3. A.~Buch earlier proved a combinatorial rule for $D_{\lambda,\mu}^{\nu}$ that we recall for
comparison purposes in Section~3.

Although the numbers $D_{\lambda,\mu}^{\nu}$ are perhaps less well-known
than $c_{\lambda,\mu}^{\nu}$ and $C_{\lambda,\mu}^{\nu}$, there is interest in their 
combinatorics. These splitting coefficients appear in the study of
quiver/degeneracy loci
of vector bundles over smooth projective varieties
\cite{Buch:quiver}.
The coefficients $D_{\lambda,\mu}^\nu$ can also be viewed
as the structure constants for a coproduct on a ring $\Gamma$ defined by
A.~Buch as a $K$-theoretic analogue of the ring of symmetric functions, see \cite{Buch:KLR} and \cite{Lam.Pyl}.
In addition, they appear (in greater generality) in formulas for
Grothendieck polynomials; see \cite{Lenart.Robinson.Sottile}, \cite{BKSTY}
and the references therein.

    In our study of the direct sum map in Section~2, we
recall the classes of \emph{Schubert boundary ideal sheaves} $[\partial X_{\lambda}]\in K^0(X)$, which form
a linear basis for $K^0(X)$ dual to $\{{\mathcal O}_{X_{\lambda}}\}$. Consequently, we discuss two other coefficients, parallel to $C_{\lambda,\mu}^{\nu}$ and $D_{\lambda,\mu}^{\nu}$ above, namely
\[[\partial X_{\lambda}][\partial X_{\mu}]=\sum_{v}E_{\lambda,\mu}^{\nu}[\partial X_{\nu}].\]
and
\[\psi^{\star}([\partial X_{\nu}])=
\sum_{\nu} F_{\lambda,\mu}^{\nu}\ [\partial Y_{\lambda}]\otimes
[\partial Z_{\mu}].\]
In general, the coefficients $C_{\lambda,\mu}^{\nu}$, $D_{\lambda,\mu}^{\nu}$ and $E_{\lambda,\mu}^{\nu}$ are
different, while A.~Buch has shown $D_{\lambda,\mu}^{\nu}=F_{\lambda,\mu}^{\nu}$.
In brief, this paper establishes that \emph{jeu de taquin} for increasing tableaux provides combinatorics for each of these four kinds of
coefficients; compare Theorems~\ref{cor:main},~\ref{thm:TY5main},~\ref{thm:main3} and Corollary~\ref{cor:main3} below.

\subsection{Main results}
We briefly recall the basic constructions of \cite{Thomas.Yong:V};
we refer the reader to that paper for further details.

An {\bf increasing tableau} $T$ is a filling of each square of
the skew shape
${\tt shape}(T)=\nu/\lambda$
with a natural number such that the rows and columns strictly increase.
Let ${\tt INC}(\nu/\lambda)$ be the set of these increasing tableaux.

A {\bf short ribbon} ${\mathfrak R}$ is a connected
skew shape that does not contain a $2\times 2$ subshape and where
each row and column contains at most two boxes. An
{\bf alternating ribbon} fills ${\mathfrak R}$ with two symbols where
adjacent boxes are filled differently. Define
${\tt switch}({\mathfrak R})$ to be the same shape as
${\mathfrak R}$ but where each box is instead filled with the other symbol,
except that  if ${\mathfrak R}$ consists of a single box,
$\tt switch$ does nothing to
it.
Define $\tt switch$ to act on a union of alternating ribbons,
by acting on each separately.
For example:

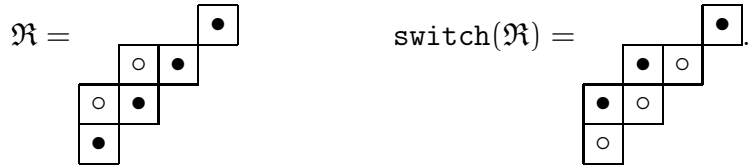
\[{\mathfrak R}=\tableau{&&&{\bullet }\\&{\circ}&{\bullet}\\{\circ}&{\bullet}\\{\bullet}}
\mbox{ \ \ \ \ \ \ \ \ \ \ \ \ \ \
${\tt switch}({\mathfrak R})=
\tableau{&&&{\bullet}\\&{\bullet}&{\circ}\\{\bullet}&{\circ}\\{\circ}}$.}
\]

Given $T\in {\tt INC}(\nu/\lambda)$,
an {\bf inner corner} is a maximally southeast box $x\in\lambda$.
Fix some set of inner corners
$\{x_1,\ldots,x_s\}$, and fill each with a ``$\bullet$''.
Let ${\mathfrak R}_1$ be the
union of alternating ribbons
made of boxes using $\bullet$ or $1$.  Apply
$\tt switch$ to ${\mathfrak R}_1$.  Now let ${\mathfrak R}_2$ be the union of alternating ribbons
using $\bullet$ or $2$, and proceed as before.
Repeat until the $\bullet$'s
have been {\tt switch}ed past all the entries of $T$.
The final placement of the numerical
entries gives  $K{\tt jdt}_{\{x_i\}}(T)$.

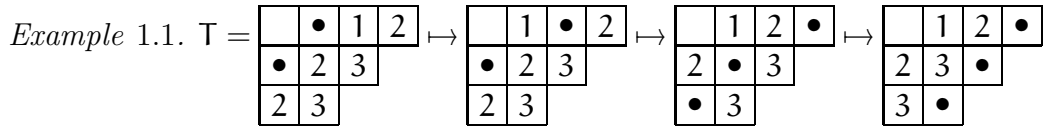
\begin{Example}
\label{exa:swseq}
$T=\tableau{{ \ }&{ \bullet }&{1}&{2}\\{ \bullet }&{2}&{3}\\{2}&{3}}\mapsto
\tableau{{ \ }&{ 1 }&{\bullet}&{2}\\{ \bullet }&{2}&{3}\\{2}&{3}}
\mapsto
\tableau{{ \ }&{ 1 }&{2}&{\bullet}\\{ 2 }&{\bullet}&{3}\\{\bullet}&{3}}
\mapsto
\tableau{{ \ }&{ 1 }&{2}&{\bullet}\\{ 2 }&{3}&{\bullet}\\{3}&{\bullet}}$

Hence
$K{\tt jdt}_{\{x_i\}}(T)=
\tableau{{ \ }&{ 1 }&{2}\\{ 2 }&{3}\\{3}}$. \qed
\end{Example}

Given $T\in {\tt INC}(\nu/\lambda)$, iterate
$K{\tt jdt}$-slides until no such moves are possible.
The result, $K{\tt rect}(T)$, which we call
\emph{a} $K$-{\bf rectification} of $T$, is an increasing
tableau of straight shape, i.e., one whose shape is given by some partition
$\mu$. The choice of $K{\tt jdt}$ slides
is a {\bf rectification order}; these orders are in natural bijection with $S\in {\tt INC}(\lambda)$,
see \cite[Section~3]{Thomas.Yong:V}.  Such rectifications are used in \cite{Thomas.Yong:V} to provide a rule
for $C_{\lambda,\mu}^{\nu}$ (see Theorem~\ref{thm:TY5main} below). In general, there can be multiple
$K$-rectifications of
a given tableau.

Our main combinatorial theorem establishes that, in contrast to
the general setting described above, if we require
$\lambda$ to be {\it rectangular}, then the $K$-rectification of an
increasing tableau of shape $\nu/\lambda$ {\it is} well-defined.
We will explain in a moment how this theorem applies to the
calculation of splitting coefficients.


\begin{Theorem}\label{thm:main}
Let $R=a\times b\subseteq \nu$.
\begin{itemize}
\item[(I)] If $T\in {\tt INC}(\nu/R)$, then
${\tt Krect}(T)$ is independent of the rectification
order.
\item[(II)] The number of tableaux $T\in {\tt INC}(\nu/R)$ that
    $K$-rectify
to a given $C\in {\tt INC}(\mu)$ does not depend on $C$.
\end{itemize}
\end{Theorem}

In Section~5 we show that for increasing tableaux, Theorem~\ref{thm:main} is sharp:
if $\lambda$ is a non-rectangular shape, one can construct a shape
$\nu\supseteq\lambda$ for which (I) does not hold, and thus
(II) is also not well-defined. This contrasts with
the classical setting where M.-P.~Sch\"{u}tzenberger's fundamental theorems guarantee that
(I) and (II) hold
for rectifications of standard Young tableaux of arbitrary skew shape $\nu/\lambda$.

To connect Theorem~\ref{thm:main} to the splitting coefficients, we
use the theorem in the following
special case:
given $\lambda$ and $\mu$ define a (skew) Young diagram
$\nu=\lambda\star\mu$ by placing $\lambda$ and $\mu$ corner to corner,
with $\lambda$ southwest of $\mu$,
as in the example below:
\begin{Example}
If $\lambda=(4,3,1)$ and $\mu=(3,2)$ then $\lambda\star\mu=
\tableau{&&&&{\ }&{\ }&{ \ }\\&&&&{\ }&{\ }\\{\ }&{\ }&{\ }&{\ }\\
{\ }&{\ }&{\ }\\{\ }}.$\qed
\end{Example}

Clearly, Theorem~\ref{thm:main} applies to $K$-rectifications of
increasing tableaux of skew shape $\lambda\star\mu$,
and thereby defines a new product
on increasing tableaux, which is in some sense an analogue of D.~Knuth's plactic product \cite{Knuth}.
In \cite{Thomas.Yong:VI}, the special case of such $K$-rectifications, where $\mu$
is a single box, plays an important role:
we prove that these $K$-rectifications compute
the Hecke (row) insertion \cite{BKSTY} of a label into an increasing tableau of shape $\lambda$. See the discussion in Section~6.

We are now ready to state our
new combinatorial rule for $D_{\lambda,\mu}^{\nu}$:

\begin{Theorem}
\label{cor:main}
Fix $U\in {\tt INC}(\nu)$.
The $K$-theory splitting coefficient $D_{\lambda,\mu}^{\nu}$
equals $(-1)^{|\lambda|+|\mu|+|\nu|}$ times the number of
$T\in {\tt INC}(\lambda\star\mu)$ such that
${\tt Krect}(T)=U$.
\end{Theorem}

If $|\lambda|+|\mu|=|\nu|$ and $U$ is a standard Young tableau, it is automatic
that the only $T\in {\tt INC}(\lambda\star\mu)$ that can contribute to
$D_{\lambda,\mu}^{\nu}(=c_{\lambda,\mu}^{\nu})$ in Theorem~\ref{cor:main}
are in fact standard Young tableaux. Thus we recover a
classical formulation of the Littlewood-Richardson rule.

We recall our theorem \cite[Theorem~1.4]{Thomas.Yong:V} for the Schubert structure constants 
$C_{\lambda,\mu}^{\nu}$. 
An increasing tableau $S_{\mu}\in {\tt INC}(\mu)$ is {\bf superstandard} if its first row
consists of the labels $1,2,\ldots,\mu_1$, second row consists of $\mu_1+1,\mu_1+2,\ldots,\mu_1+\mu_2$
etc.

\begin{Theorem}
\label{thm:TY5main}
The Schubert structure constant $C_{\lambda,\mu}^{\nu}$ equals $(-1)^{|\lambda|+|\mu|+|\nu|}$ times the
number of $T\in {\tt INC}(\nu/\lambda)$ such that ${\tt Krect}(T)=S_{\mu}$.
\end{Theorem}


In the classical setting of \cite{Schutzenberger}
one can apply (\ref{eqn:conflate}) to automatically obtain a \emph{jeu
de taquin} rule for the cohomology splitting coefficients. However,
in $K$-theory, one lacks this
coincidence, and is forced to see two different fundamental theorems,
one relevant for the product (as presented in \cite{Thomas.Yong:V}; see also Theorem~\ref{thm:welldefined}),
and one relevant for the splitting coefficients (Theorem~\ref{thm:main}). This underscores our
hope to bring further insight into Sch\"{u}tzenberger's \emph{classical} {\it jeu de taquin},
by understanding it inside the larger, more flexible setting of {\it jeu de taquin}
for increasing tableaux.

We now turn to our results concerning Schubert boundary ideal sheaves.

Define a superset ${\widehat {\tt INC}(\nu/\lambda)}$ of
increasing tableaux, by filling $\nu/\lambda$
with positive integers and the label ``$X$'' on some number (possibly zero)
of the outer corners of $\nu/\lambda$. We still demand
the rows and columns to strictly increase; the labels $X$
can be assumed to have value $\infty$ insofar as this condition is concerned.

For example, we have $\tableau{&&{1}&{2}&{X}\\&{1}&{X}\\{2}&{4}}\in {\widehat {\tt INC}}((5,3,2)/(2,1))$.

The following combinatorial rule for computing $E_{\lambda,\mu}^{\nu}$ is
``positive'' in the sense of D.~Anderson, S.~Griffeth
and E.~Miller \cite{Anderson.Griffeth.Miller}. It is the first such rule for these coefficients (for general $\lambda,\mu,\nu$)
that we are aware of:

\begin{Theorem}
\label{thm:main3}
The structure constant $E_{\lambda,\mu}^{\nu}$ equals $(-1)^{|\nu|-|\lambda|-|\mu|}$ times the number of tableaux $T\in {\widehat {\tt INC}}(\nu/\lambda)$ such that $K{\tt rect}(T)=S_{\mu}$, where the $X$'s are erased before computing the
$K$-rectification.
\end{Theorem}

Finally, using A.~Buch's observation $F_{\lambda,\mu}^{\nu}=D_{\lambda,\mu}^{\nu}$ (cf. Section~2.4),
one obtains an additional consequence of Theorems~\ref{thm:main} and~\ref{cor:main}:

\begin{Corollary}
\label{cor:main3}
$F_{\lambda,\mu}^{\nu}$ is computed by the same combinatorial rule as for $D_{\lambda,\mu}^{\nu}$ from Theorem~\ref{cor:main}.
\end{Corollary}

\subsection{Organization}
In Section~2, we investigate the geometry of the direct sum map and give a new
proof of the basic identity we use to relate the $K$-theory splitting
coefficients to the
$K$-theory product. This involves a discussion of
\emph{Schubert boundary ideal
sheaves}, which we then use to prove Theorem~\ref{thm:main3} and Corollary~\ref{cor:main3}.
In Section~3, we prove Theorems~\ref{thm:main} and~\ref{cor:main}. Section~4
contains examples. Section~5 explains why Theorem~\ref{thm:main} is sharp.
Lastly, in Section~6 we discuss a new product on increasing tableaux.

\section{The direct sum map, Richardson varieties, and boundary ideal sheaves}

\subsection{The image of the direct sum map}
We now discuss the geometry of the direct sum map from Section~1.1, which
defines the splitting coefficients. We show that the image of any product of Schubert varieties
is a Richardson variety and apply this fact to geometrically interpret and prove an identity we
need in the proof of Theorem~\ref{cor:main}.

It is convenient for this purpose to identify a Young diagram $\lambda\subseteq \Lambda$
with a subset
\[{\overline \lambda}=\{{\overline \lambda}_1<{\overline \lambda}_2<\ldots<{\overline \lambda}_k\} \subseteq \{1,2,3,\ldots,n\}.\]
Up to a choice of
conventions, the bijection is standard: we trace the boundary of $\lambda$ sitting in $\Lambda$ by walking
from the northeast corner to the southwest corner of $\Lambda$. We include the integer $i$ in
${\overline \lambda}$ if the $i$-th step of the above walk is a ``down'' step.

Fix a choice of reference flags
\[A_{\bullet}: A_0\subset A_1\subset \ldots \subset A_{n_1}\cong {\mathbb C}^{n_1}\]
where $A_i=\langle e_1,e_2,\ldots,e_i\rangle$ and $\{e_1,\ldots, e_{n_1}\}$
is the standard basis of ${\mathbb C}^{n_1}$, and
\[B_{\bullet}: B_0\subset B_1\subset \ldots \subset B_{n_2}\cong {\mathbb C}^{n_2},\]
where $B_j=\langle f_1,f_2,\ldots, f_{j}\rangle$ and $\{f_1,\ldots,f_{n_2}\}$
is the standard basis of ${\mathbb C}^{n_2}$.

Now define corresponding {\bf Schubert varieties}
\[Y_{\lambda}=\left\{W\in Y\  | \ \dim(W\cap A_{{\overline \lambda}_t})\geq t,\ 1\leq t\leq k_1\right\}\subseteq Y\]
and
\[Z_{\mu}=\left\{W\in Z\  | \ \dim(W\cap B_{\overline\mu _t})\geq t,\ 1\leq t\leq k_2\right\}\subseteq Z.\]
Our conventions imply that the dimensions of $Y_{\lambda}$ and
$Z_{\mu}$ are $|\lambda^{\vee}|$ and
$|\mu^{\vee}|$ respectively. Here $\lambda^{\vee}$ is the 180 degree rotation of the complement of $\lambda$ in the rectangle
$k_1\times (n_1-k_1)$, while one defines $\mu^{\vee}$ similarly, with respect to the rectangle $k_2\times (n_2-k_2)$.

Now consider the following reference flag $C_{\bullet}$
for ${\mathbb C}^{n_1}\oplus {\mathbb C}^{n_2}\cong {\mathbb C}^{n_1+n_2}$
\begin{multline}\nonumber
C_0=A_0\oplus B_0 \subset C_1=A_1\oplus B_0 \subset C_2=A_2\oplus B_0 \subset \ldots \subset C_n=A_n\oplus B_0 \cong {\mathbb C}^{n_1}\\ \nonumber \subset C_{n+1}=A_n\oplus B_1 \subset C_{n+2}=A_n\oplus B_2 \subset \ldots\subset
                                        C_{n+m}=A_n\oplus B_m\cong {\mathbb C}^{n_1+n_2},
\end{multline}
which is the standard flag with respect to the following basis of ${\mathbb C}^{n_1}\oplus {\mathbb C}^{n_2}$:
\begin{equation}
\label{eqn:thebasis}
\{e_1\oplus {\vec 0}, e_2\oplus {\vec 0},\ldots, e_{n_1}\oplus {\vec 0},
{\vec 0}\oplus f_1,{\vec 0}\oplus f_2,\ldots, {\vec 0}\oplus f_{n_2}\}.
\end{equation}
Similarly, let $X_{\nu}$ be the Schubert subvariety of $X\cong Gr(k_1+k_2,{\mathbb C}^{n_1}\oplus {\mathbb C}^{n_2})$
indexed by $\nu$ with respect to the above reference flag.

Let $\mu\oslash\lambda$ be the straight shape
Young diagram obtained by placing $\lambda$ to the right
of the $k_1\times (n_2-k_2)$ rectangle, and $\mu$ below the said rectangle. (Compare with
the definition of the skew shape $\mu\star\lambda$ given after Theorem~\ref{thm:main}.)
Note that the
resulting shape fits inside $\Lambda=k\times(n-k)=(k_1+k_2)\times (n_1+n_2-k_1-k_2)$, i.e., it indexes
a Schubert variety of $X$. Also,
\[\overline{\mu\oslash\lambda}= {\overline \lambda_1}<{\overline \lambda_2}<\ldots
<{\overline \lambda_{n_1}}<n_1+{\overline \mu_1}<n_1+{\overline \mu_2}<\ldots < n_1+
{\overline \mu_{n_2}}.\]
Hence, we have
\[X_{\mu\oslash\lambda}=\left\{W\in X \ | \ \dim(W\cap C_{({\overline{\mu\oslash\lambda}})_{t}})\geq t, 1\leq t\leq k\right\}.\]

We also need the {\bf opposite Schubert variety} $X^{\nu}$,
which is the Schubert variety defined to the
{\bf opposite reference flag}
\[\langle \vec 0\oplus \vec 0\rangle\subset
C_{1}^{opp}\subset C_{2}^{opp}\subset \ldots\subset C_{n_1+n_2}^{opp}\cong {\mathbb C}^{n_1+n_2}\]
where $C_{j}^{opp}$ is the subspace spanned by the \emph{last} $j$ vectors in the basis
(\ref{eqn:thebasis}), and defined with respect to the subset
\[\{n-{\overline{\nu}}_{k-t+1} +1 \ | \ 1\leq t\leq k\}\subseteq \{1,2,\ldots,n\}\]
(this is the subset obtained by reading the boundary of $\nu$ from the southwest to the northeast).
That is,
\[X^{\nu}=\{W\in X\ | \ \dim(W\cap C^{opp}_{n-{\overline{\nu}}_{k-t+1} +1})\geq t, 1\leq t\leq k\}.\]

In particular, we will interested in the opposite Schubert variety for the Young diagram
\[\Omega=((n-k)^{k_1},(n_2-k_2)^{k_2}),\]
i.e., the shape obtained by removing
the southeast $k_2\times (n_1-k_1)$ subrectangle from $\Lambda$.

In general, the intersection $X_{\gamma}\cap X^{\delta}$ is nonempty if and only if and only if $\gamma\subseteq \delta$.
In that case, the intersection is reduced and irreducible, and is called the
{\bf Richardson variety} $X_{\gamma}^{\delta}$. It is also true that
\begin{equation}
\label{eqn:richardsondim}
\dim(X_{\gamma}^{\delta})=|\delta|-|\gamma|.
\end{equation}

We are now ready to record the following fact, that shows
$\psi$ maps a product of Schubert
varieties isomorphically to a Richardson variety.
\begin{Proposition}
\label{prop:directsumgeometry}
The direct sum map induces an isomorphism
\[\psi:Y_{\lambda}\times Z_{\mu}\subseteq Y\times Z\  \tilde{\rightarrow} \ X_{\mu\oslash\lambda}^{\Omega}\subseteq X.\]
\end{Proposition}
\begin{proof}
The direct sum map is clearly injective, with image contained inside
$X_{\mu\oslash\lambda}$, by the definitions. Similarly, the
image is contained inside $X^{\Omega}$ since \emph{a fortiori} even the image of
$Y\times Z$ is. Moreover,
$\dim(Y_{\lambda}\times Z_{\mu})=k_1(n_1-k_1)-|\lambda| + k_2(n_2-k_2)-|\mu|$,
which is the dimension of $X_{\mu\oslash\lambda}^{\Omega}$, by (\ref{eqn:richardsondim}).
Finally, since any Richardson
variety is irreducible, the
isomorphism claim follows.
\end{proof}

\subsection{Implications in cohomology}
We will now use the isomorphism we have just established to obtain an identity
between the cohomological product coefficients and the splitting coefficients.
We use some basic facts from (co)homology theory on general algebraic varieties,
see, e.g., \cite[Appendix B]{Fulton}, and on flag varieties in particular, see, e.g.,
\cite{Brion:notes} and \cite{Fulton}.

Proposition~\ref{prop:directsumgeometry} implies the Gysin homomorphism pushforward $\psi_{\star}$ on cohomology
\[\psi_{\star}:H^{\star}(Y)\otimes H^{\star}(Z)(\cong H^{\star}(Y\times Z))
  \to   H^{\star}(X)\]
is given by
\begin{eqnarray}\nonumber
[Y_{\lambda}]\otimes [Z_{\mu}] \mapsto [X_{\mu\oslash\lambda}^{\Omega}] & = &
[X_{\mu\oslash\lambda}]\cup [X^{\Omega}]\\ \nonumber
& = & [X_{\mu\oslash\lambda}]\cup [X_{\Omega^{\vee}}]\\ \nonumber
& = &\sum_{\nu^{\vee}}c_{\mu\oslash\lambda, \Omega^{\vee}}^{\nu^{\vee}}[X_{\nu^{\vee}}].\nonumber
\end{eqnarray}
From this pushforward, we can deduce the pull-back map $\psi^{\star}$ on cohomology.

The cohomology class
\[[X_{\nu}]\in H^{2|\nu|}(X)\]
is Poincar\'{e} dual to the cohomology class
\[[X_{\nu^{\vee}}]=[X^{\nu}]\in H^{\dim(X)-2|\nu|}(X).\]
Similarly, the cohomology class
\[[Y_{\lambda}]\otimes [Z_{\mu}]\in H^{2|\lambda|}(Y)\otimes H^{2|\mu|}(Z)\]
is Poincar\'{e} dual to
\[[Y_{\lambda^{\vee}}]\otimes [Z_{\mu^{\vee}}]\in H^{\dim(Y)-2|\lambda|}(Y)\otimes H^{\dim(Z)-2|\mu|}(Z).\]

Recall the {\bf projection formula}: for any continuous map $g:B\to A$ we have
\begin{equation}
\label{eqn:cohomproj}
g_{\star}(g^{\star}(\alpha)\cup \beta)=\alpha\cup g_{\star}(\beta),
\end{equation}
where $\alpha\in H^i(A)$ and $\beta\in H^{j}(B)$, and here $g_{\star}$ is the Gysin homomorphism
pushing-forward cohomology.

We will apply (\ref{eqn:cohomproj}) in the case
\[B=Y\times Z, \ \ A=X, \ \ g=\psi, \ \ \alpha=[X_{\nu}], \ \
\beta=[Y_{\lambda^{\vee}}]\otimes [Z_{\mu^{\vee}}].\]
This gives an equality of multiples of the point class $[pt]\in H^{\dim(X)}(X)$, and therefore if
\[\psi^{\star}([X_{\nu}])=\sum_{\nu}d_{\lambda,\mu}^{\nu}
[Y_{\lambda}]\otimes [Z_{\mu}]\]
where $d_{\lambda,\mu}^{\nu}$ is the (cohomological) splitting coefficient, then we have
\begin{equation}
\label{eqn:homologyequality}
d_{\lambda,\mu}^{\nu}=c_{\mu^{\vee}\oslash\lambda^{\vee}, \Omega^{\vee}}^{\nu^{\vee}}
=c_{\Omega^{\vee},\nu}^{(\mu^{\vee}\oslash\lambda^{\vee})^{\vee}}.
\end{equation}
Here the last equality follows from the $S_3$-symmetry of the Littlewood-Richardson coefficients.
Now, $(\mu^{\vee}\oslash\lambda^{\vee})^{\vee}$ is precisely the shape obtained by placing
$\lambda$ below, and $\mu$ to the right of $\Omega^{\vee}(=k_2\times (n_1-k_1))$. If we call this shape
$\lambda\dagger\mu$ (compare with the definition of $\oslash$ in Section~2.1) then we obtain
\begin{equation}
\label{eqn:dagger}
d_{\lambda,\mu}^{\nu}=c_{\Omega^{\vee},\nu}^{\lambda\dagger\mu}.
\end{equation}
On its face, this formula expresses an arbitrary cohomological splitting coefficient
as a classical Littlewood-Richardson coefficient.
Alternatively, if we accept the truth of the fact that
that $d_{\lambda,\mu}^{\nu}=c_{\lambda,\mu}^{\nu}$, we can view this equation
as expressing a known identity
among the classical Littlewood-Richardson coefficients.

\subsection{Extension to $K$-theory}

The argument of the previous subsection to deduce (\ref{eqn:homologyequality}) can be extended to
$K$-theory, but we separate them in order to emphasize salient differences between the two cases. We need some facts about $K$-theory of flag varieties due to A.~Knutson. These were reported in \cite{Buch:KLR}
and also are part of the course notes \cite{Brion:notes}.

Let
\[\rho_{Y}:Y\to \{\star\}\]
be the map to a point. The pushforward
\[(\rho_{Y})_{\star}:K^{0}(Y)\to K^{0}(\{\star\})\cong {\mathbb Z}\]
on $K$-cohomology (as identified with $K$-homology) defines a perfect pairing
\[\langle p,q\rangle_{Y}=(\rho_{Y})_{\star}(p\cdot q),  \ \mbox{ \ for \ } p,q\in K^{0}(Y).\]

The same definition can be made in cohomology.  The Schubert basis for
cohomology is self-dual with respect to this pairing.
Contrastingly, in $K$-theory, there are two natural bases which are dual
to each other: the basis of structure sheaves, $[{\mathcal O}_{X_\lambda}]$,
and the basis of ideal sheaves of the boundaries defined by:
$\partial X_{\lambda}:=X_{\lambda}\setminus X_{\lambda}^{\circ}$,
where $X_{\lambda}^{\circ}$ is the open Schubert cell indexed by $\lambda$. Specifically, $[{\mathcal O}_{X_{\lambda}}]$ is dual to
$[\partial X_{\lambda^{\vee}}]$.

The relation between the two bases is given by
\begin{equation}
\label{eqn:basisrelation}
[\partial X_{\lambda}]=(1-[{\mathcal O}_{X_{(1)}}])[{\mathcal O}_{X_{\lambda}}]
=\sum_{\lambda\mapsto {\widetilde \lambda}}
(-1)^{|{\widetilde\lambda}/{\lambda}|}[{\mathcal O}_{X_{{\widetilde \lambda}}}]\in K^0(X_1)
\end{equation}
where $\lambda\mapsto {\widetilde \lambda}$ means that ${\widetilde \lambda}/\lambda$ is a skew shape consisting of some (non-negative) number of boxes,
no two of which are in the same row or column.

The inverse relation is:
\[[{\mathcal O}_{X_{\lambda}}]=\sum_{\lambda\subseteq \alpha} [\partial X_{\alpha}].\]

Combining Proposition~\ref{prop:directsumgeometry} and \cite[Section~3.4 \& Lemma~4.1.1]{Brion:notes},
it follows that the pushforward on $K$-cohomology is defined by
\begin{eqnarray}
\label{eqn:pushforward}
\psi_{\star}([{\mathcal O}_{Y_{\lambda}}]\otimes [{\mathcal O}_{Z_{\mu}}]) & = &
[{\mathcal O}_{X_{\mu\oslash\lambda}\cap X^{\Omega}}]\\ \nonumber
& = & [{\mathcal O}_{X_{\mu\oslash\lambda}}][{\mathcal O}_{X^{\Omega}}]\\ \nonumber
& = & [{\mathcal O}_{X_{\mu\oslash\lambda}}][{\mathcal O}_{X_{\Omega^{\vee}}}]\\ \nonumber
& = & \sum_{\nu^{\vee}}
C_{\mu\oslash\lambda, \Omega^{\vee}}^{\nu^{\vee}}[{\mathcal O}_{X_{\nu^{\vee}}}].\nonumber
\end{eqnarray}
One still has a projection formula, valid for any proper morphism $g:B\to A$, such as $\psi$:
\[g_{\star}((g^{\star}\alpha)\cdot\beta)=\alpha\cdot g_{\star}\beta.\]
Moreover, we are interested in hitting both sides with $\rho_{\star}$:
\begin{equation}
\label{eqn:Kproj}
\rho_{\star}\circ g_{\star}((g^{\star}\alpha)\cdot\beta)=\rho_{\star}(\alpha\cdot g_{\star}\beta).
\end{equation}

Applying the projection formula (\ref{eqn:Kproj}) with
\[B=Y\times Z, \ \ A=X, \ \ g=\psi, \ \ \alpha=[{\mathcal O}_{X_{\nu}}],  \ \ \beta=[\partial{Y_{\lambda^{\vee}}}]\otimes [\partial{Z_{\mu^{\vee}}}],\]
we obtain
\begin{equation} \label{explicit.eq}
((\rho_X)_\star\circ \psi_\star)(\psi^\star(\mathcal [{\mathcal O}_{X_\nu}]) \cdot [\partial{Y_{\lambda^{\vee}}}]\otimes [\partial{Z_{\mu^{\vee}}}]) = (\rho_X)_\star(\mathcal [{\mathcal O}_{X_\nu}]\cdot \psi_\star([\partial{Y_{\lambda^{\vee}}}]\otimes [\partial{Z_{\mu^{\vee}}}])
\end{equation}

Note that the map $\rho_{Y\times Z}:Y\times Z \to \{\star\}$
clearly factors as $\rho_{X}\circ \psi$. Thus, the pushforward $(\rho_{X})_{\star}\circ \psi_{\star}$
defines an inner product for which
\[\{[{\mathcal O}_{Y_{\lambda}}]\otimes [{\mathcal O}_{Z_{\mu}}]\}
\mbox{ \ and \ }
\{[\partial{Y_{\lambda^{\vee}}}]\otimes [\partial{Z_{\mu^{\vee}}}]\}\]
are dual bases of
\[K^{0}(Y\times Z)\cong K^{0}(Y)\otimes K^{0}(Z).\]
With this in mind, it is automatic that
the left hand side of (\ref{explicit.eq}) equals $D_{\lambda,\mu}^{\nu}$.

To compute the righthand side of (\ref{explicit.eq}) we use
\begin{Lemma}
\label{lemma:pushforwardfact}
\[\psi_{\star}(\beta)=[\partial{X_{\mu^{\vee}\oslash\lambda^{\vee}}}][{\mathcal O}_{X_{\Omega^{\vee}}}].\]
\end{Lemma}
\begin{proof}
First note that by applying (\ref{eqn:basisrelation}) to $\beta$ we have
\begin{eqnarray}\nonumber
\beta=[\partial{Y_{\lambda^{\vee}}}]\otimes [\partial{Z_{\mu^{\vee}}}] &
= & \left(\sum_{\lambda^\vee\mapsto {\widetilde {\lambda^{\vee}}}}
(-1)^{|{\widetilde{\lambda^{\vee}}}/{\lambda^{\vee}}|}[{\mathcal O}_{Y_{{\widetilde {\lambda^{\vee}}}}}]\right)
\otimes
\left(\sum_{\mu^\vee\mapsto {\widetilde {\mu^{\vee}}}}
(-1)^{|{\widetilde{\mu^{\vee}}}/{\mu^{\vee}}|}[{\mathcal O}_{Z_{{\widetilde {\mu^{\vee}}}}}]\right)\\ \nonumber
& = & \sum_{\lambda^\vee\mapsto \widetilde{\lambda^{\vee}}, \ \mu^\vee\mapsto {\widetilde{\mu^{\vee}}}}
(-1)^{|{\widetilde{\lambda^{\vee}}}/{\lambda^{\vee}}|+|{\widetilde{\mu^{\vee}}}/{\mu^{\vee}}|}
[{\mathcal O}_{Y_{{\widetilde {\lambda^{\vee}}}}}]\otimes [{\mathcal O}_{Z_{{\widetilde {\mu^{\vee}}}}}].
\end{eqnarray}
Next, applying (\ref{eqn:pushforward}) to the above expression for $\beta$ gives:
\[\psi_{\star}(\beta)=\sum_{\lambda^{\vee}\mapsto \widetilde{\lambda^{\vee}}, \ \mu^{\vee}\mapsto {\widetilde{\mu^{\vee}}}}
(-1)^{|{\widetilde{\lambda^{\vee}}}/{\lambda^{\vee}}|+|{\widetilde{\mu^{\vee}}}/{\mu^{\vee}}|}
\sum_{\nu^{\vee}}
C_{{\widetilde{\mu^{\vee}}}\oslash {\widetilde{\lambda^{\vee}}}, {\Omega^{\vee}}}^{\nu^{\vee}}
[{\mathcal O}_{X_{\nu^{\vee}}}].
\]
However, the summation is equal to
\[\sum_{\mu^\vee\oslash\lambda^\vee\mapsto \widetilde{\mu^\vee\oslash\lambda^\vee}}
(-1)^{|\widetilde{\mu^\vee\oslash\lambda^\vee}/\mu^\vee\oslash\lambda^\vee|}
\sum_{\nu^{\vee}}
C_{\widetilde{\mu^\vee\oslash\lambda^\vee},\Omega^{\vee}}^{\nu^{\vee}}[{\mathcal O}_{X_{\nu^{\vee}}}],
\]
where we have used the fact that $C_{\gamma,\delta}^{\kappa}=0$ if $\gamma$ and the 180 degree rotation of
$\delta$ overlap in~$\Lambda$.

On the other hand, the latter expression is precisely equal to the result of
expanding $[\partial{X_{\mu^{\vee}\oslash\lambda^{\vee}}}][{\mathcal O}_{X_{\Omega^{\vee}}}]$,
after applying
 (\ref{eqn:basisrelation}) to $[\partial{X_{\mu^{\vee}\oslash\lambda^{\vee}}}]$.
\end{proof}

Thus, the right hand side of (\ref{explicit.eq}) is:
\[(\rho_X)_\star([{\mathcal O}_{X_\nu}]\cdot \psi_\star([\partial{Y_{\lambda^{\vee}}}]\otimes [\partial{Z_{\mu^{\vee}}}])
=\rho_{\star}([{\mathcal O}_{X_{\nu}}][\partial{X_{\mu^{\vee}\oslash\lambda^{\vee}}}][{\mathcal O}_{X_{\Omega^{\vee}}}])
=C_{\Omega^\vee,\nu}^{(\mu^{\vee}\oslash\lambda^{\vee})^{\vee}}=C_{\Omega^{\vee},\nu}^{\lambda\dagger\mu}\]
where $\lambda\dagger\mu$ is defined just before (\ref{eqn:dagger}).

Summarizing, we have geometrically explained the following identity:
\begin{equation}
\label{eqn:prodcoprodreln}
D_{\lambda,\mu}^{\nu}=C_{\Omega^\vee,\nu}^{\lambda\dagger\mu}.
\end{equation}
Up to minor differences, this identity was earlier obtained
by A.~Buch \cite{Buch:KLR}, via the combinatorics of Grothendieck polynomials. It should also be said that versions of Proposition~\ref{prop:directsumgeometry} and Lemma~\ref{lemma:pushforwardfact} were proved
in \cite{Lenart.Robinson.Sottile} in the study of an analogue of the direct sum map on a product of flag varieties to a larger
flag variety. They obtain a class of identities among $K$-theoretic flag structure constants, which can be thought
of as generalizations of~(\ref{eqn:prodcoprodreln}).

\subsection{Proofs of Theorem~\ref{thm:main3} and Corollary~\ref{cor:main3}}
We stated that $[\partial X_{\lambda}]=(1-[{\mathcal O}_{X_{(1)}}])[{\mathcal O}_{X_{\lambda}}]$.
Hence
\[[\partial X_{\lambda}][\partial X_{\mu}]=[{\mathcal O}_{X_{\lambda}}][{\mathcal O}_{X_{\mu}}](1-[{\mathcal O}_{X_{(1)}}])^2
=\sum_{\overline\nu}C_{\lambda,\mu}^{\overline\nu}\left([{\mathcal O}_{X_{\overline\nu}}]\cdot (1-[{\mathcal O}_{X_{(1)}}])^2\right).\]
Now, we have
\[(1-[{\mathcal O}_{X_{(1)}}])[{\mathcal O}_{X_{\overline\nu}}]=\sum_{\overline\nu\mapsto {\nu}}(-1)^{|{\nu}|-|\overline\nu|}[{\mathcal O}_{X_{{\nu}}}].\]

It then follows from the previous two equations that
\[[\partial X_{\lambda}][\partial X_{\mu}]=\sum_{\overline\nu\subseteq \Lambda}C_{\lambda,\mu}^{\overline\nu}\sum_{\overline\nu\mapsto {\nu}}
(-1)^{|{\nu}|-|\overline\nu|}[\partial X_{{\nu}}].\]
Therefore,
\[E_{\lambda,\mu}^{\nu}=\sum_{{\overline \nu}\mapsto \nu}C_{\lambda,\mu}^{{\overline \nu}}(-1)^{|\nu/{\overline \nu}|}.\]
Now $C_{\lambda,\mu}^{\overline \nu}$ equals $(-1)^{|{\overline \nu}|-|\lambda|-|\mu|}$ times
the number of $T\in {\tt INC}({\overline \nu}/\lambda)$ that $K$-rectify to the superstandard tableau $S_{\mu}$.
Thus
$E_{\lambda,\mu}^{\nu}=(-1)^{|\nu|-|\lambda|-|\mu|}\sum_{{\overline \nu}\mapsto \nu}|{C}_{\lambda,\mu}^{{\overline \nu}}|$.
However, the union of the tableaux that contribute to some $C_{\lambda,\mu}^{\overline \nu}$ with ${\overline \nu}\mapsto \nu$
are in obvious bijection with the tableaux of ${\widehat {\tt INC}}(\nu/\lambda)$, hence Theorem~\ref{thm:main3} follows.

The above argument can be extended to
minuscule $G/P$. One can thereby deduce a conjectural rule for
multiplying Schubert boundary ideal sheaves for these spaces, from the
$K$-theory minuscule conjecture of \cite{Thomas.Yong:V}. This point is actually used
in \cite{Thomas.Yong:X} where we deduce such a rule for odd orthogonal Grassmannians $OG(n,2n+1)$.

A.~Buch (private communication) noticed the following sequence of equalities:
\begin{eqnarray}
\psi^{\star}([\partial X_{\nu}]) & = & \psi^{\star}\left((1-[{\mathcal O}_{X_{(1)}}])[{\mathcal O}_{X_{\nu}}]\right)\\\nonumber
& = & \psi^{\star}((1-[{\mathcal O}_{X_{(1)}}]))\psi^{\star}([{\mathcal O}_{X_{\nu}}])\\\nonumber
& = & \left((1-[{\mathcal O}_{Y_{(1)}}])\otimes (1-[{\mathcal O}_{Z_{(1)}}])\right)
\left(\sum_{\lambda,\mu}D_{\lambda,\mu}^{\nu}[{\mathcal O}_{Y_{\lambda}}]\otimes
[{\mathcal O}_{Z_{\mu}}]\right)\\\nonumber
& = & \sum_{\lambda,\mu} D_{\lambda,\mu}^{\nu}[\partial Y_{\lambda}]\otimes [\partial Z_{\mu}].
\end{eqnarray}
Corollary~\ref{cor:main3} then follows from this observation, together with Theorems~\ref{thm:main} and~\ref{cor:main}.

\section{Proof of Theorems~\ref{thm:main} and~\ref{cor:main}}

\subsection{Strong $K$-dual equivalence for rectangles}

It will be useful to speak of {\bf reverse
$K{\tt jdt}$ slides}, obtained by moving into {\bf outer corners} of a skew shape
$\Lambda$, i.e., maximally northwest boxes of $\Lambda\setminus \nu$.
These slides can be thought of ordinary ``forward'' $K{\tt jdt}$ slides after rotating $\Lambda$
and shapes contained therein $180$-degrees and reversing the order on the
tableau's entries. Hence all claims and definitions involving $K{\tt jdt}$ have obvious
reverse analogues. With this said, we now will generically refer to $K{\tt jdt}$ as meaning either
forward or reverse slides, distinguishing them only when necessary by
referring explicitly to
${\tt rev}K{\tt jdt}$. See \cite[Section~1.1]{Thomas.Yong:V}.

$A,B\in {\tt INC}(\lambda)$ are {\bf $K$-dual equivalent}
if for any common sequence of $K{\tt jdt}$ slides
\begin{equation}
\label{eqn:Aseq}
A_0:=A\mapsto A_1\mapsto A_2 \mapsto \cdots \mapsto A_M
\end{equation}
\begin{equation}
\label{eqn:Bseq}
B_0:=B\mapsto B_1\mapsto B_2 \mapsto \cdots \mapsto B_M
\end{equation}
applied to $A$ and $B$ respectively results in
${\tt shape}(A_i)={\tt shape}(B_i) \mbox{\ for $0\leq i\leq M$.}$
If instead $A,B\in {\tt SYT}(\lambda)$ and one only uses ordinary
${\tt jdt}$ slides, then this would be M.~Haiman's definition of
(ordinary) dual equivalence \cite{Haiman}.

Consider an increasing tableau $T$, and a sequence of
slides applied to it. Instead of recording only the tableaux
which result from each successive slide, as in (\ref{eqn:Aseq}) and (\ref{eqn:Bseq}),
we trace the process in slow motion, and consider what we call the {\bf switch sequence}, that
is to say, the sequence of tableaux
\begin{equation}
\label{eqn:firsttabseq}
T=T_0 \mapsto T_1 \mapsto T_2\mapsto \dots
\end{equation}
where $T_i$ and $T_{i+1}$ are related by a single {\tt switch}. Example~\ref{exa:swseq} shows one
such {\tt switch} sequence, expanding upon a single $K{\tt jdt}$ slide. Also, we will discuss
the {\bf configuration} of a tableau appearing in a {\tt switch} sequence,
which is the shape given by the boxes filled with either numerical labels or $\bullet$'s.

\begin{Example}
Among the three tableaux
$\tableau{&&{1}&{\bullet}\\&{\bullet}&{3}\\{1}&{2}}, \
\tableau{&&{1}&{\bullet}\\&{\bullet}&{2}\\{2}&{3}},
 \mbox{\ \ and}
\tableau{&&{\bullet}&{1}\\&{\bullet}&{3}\\{1}&{2}},$
the first two have the same configuration, which is different from that of the third.\qed
\end{Example}

We now define a more demanding property then $K$-dual
equivalence. We define two tableaux $A$ and $B$ to be {\bf strongly $K$-dual equivalent} if,
when a common sequence of {\tt Kjdt} slides is applied to $A$ and $B$,
the configurations of  $A_i$ and $B_i$ are the same, for all $i$.

M.~Haiman \cite{Haiman} proved that $A,B\in {\tt SYT}(\lambda)$ are always
dual equivalent, provided $\lambda$ is a straight shape (and we only use
ordinary {\tt jdt} slides). However:

\begin{Example} While
$\tableau{{1}&{2}\\{3}} \mbox{\  and \ }\tableau{{1}&{3}\\{2}}$
are dual equivalent, they are not strongly dual equivalent, since we have the {\tt switch}es
\begin{equation}
\label{eqn:notstrong}
\tableau{{1}&{2}\\{3}&{\bullet}}\mapsto \tableau{{1}&{2}\\{\bullet}&{3}}\mbox{\ and \ }
\tableau{{1}&{3}\\{2}&{\bullet}}\mapsto\tableau{{1}&{\bullet}\\{2}&{3}},
\end{equation}
resulting in tableaux with different configurations.\qed
\end{Example}

On the other hand,
we saw in \cite{Thomas.Yong:V} that in general
$A,B\in {\tt INC}(\lambda)$ need not be $K$-dual equivalent.
In contrast,
our key technical observation is the following result:
\begin{Theorem}
\label{thm:dualequiv}
All tableaux in ${\tt INC}(\lambda)$ are strongly $K$-dual equivalent
if and only if $\lambda$ is a rectangle $c\times d$.
\end{Theorem}

\begin{proof}
The proof of the ``$\Rightarrow$'' direction is to extend the example (\ref{eqn:notstrong}). If $\lambda$ is not rectangular,
consider any outer corner of $\lambda$ in a row $r$ where
$\lambda_{r-1}>\lambda_r$, as depicted in Figure~\ref{fig:forward} below (for $r=3$).
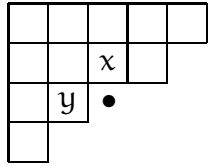
\begin{figure}[h]
\begin{picture}(100,50)
\put(40,36){\tableau{{\ }&{\ }&{\ }&{\ }&{\ }\\{\ }&{\ }&{x }&{ \ }\\{\ }&{
y }\\{ \ }}}
\put(75,11){$\bullet$}
\end{picture}
\caption{\label{fig:forward} Proof of ``$\Rightarrow$'' direction of Theorem~\ref{thm:dualequiv}}
\end{figure}

One can construct $A,B\in {\tt SYT}(\lambda)$ so that $A$ has
a smaller label in box ``$x$'' than it does in box ``$y$''
while $B$ has a smaller label in
box ``$y$'' than it does in box ``$x$''.
Then choose the $K{\tt jdt}$ slide moving into
the outer corner marked by $\bullet$ above. The first {\tt switch} which causes
any movement to take place will result in the $\bullet$ moving to different
squares in the two tableaux, so they are not strongly $K$-dual equivalent.

We now turn to the converse ``$\Leftarrow$'' direction.
This will require some preparation.

We wish to assign a {\bf box of origin} to each box $x$ of $T_i$ from
(\ref{eqn:firsttabseq}), denoted ${\mathfrak B}_i(x)$, which is a box of $T$. Speaking informally for the moment,
the idea is that, for a box $x$ of $T_i$, its box of origin is defined by
setting $\bor_i(x)$ to be the box of $T$ that the label in box $x$ evolved from through
generations of {\tt switch} operations. If all the labels of $T$ were different, then it would clearly be  possible
to do this: simply define the box of origin of $x$ to be the box of
$T$ containing the same label as $x$.  However, since $T$ may have repeated
entries, to make ${\mathfrak B}_i(x)$ well-defined, we must be more careful, as done below.

Formally, we will define the box of origin inductively.
For $T_0$, each box is its
own box of origin.
Assume the box of origin $\bor_{i-1}(x)$ is well-defined for each
box $x$ of $T_{i-1}$.  We say that the {\tt switch} from $T_{i-1}$ to
$T_i$ is {\bf uniform} if, for each short ribbon involved in the
{\tt switch}, the boxes containing numbers all have the same box of
origin.  In this case, for each short ribbon,
we define the box of origin of the {\tt switch}ed numbers
to be the same as the box of origin of the numbers before the {\tt switch}.
In general, there is nothing to guarantee that a sequence of {\tt switch}es
will all be uniform.  After a non-uniform {\tt switch}, the box of origin is
no longer well-defined. However, in the case that $T$ is of rectangular shape, we claim
this failure never occurs.

Let $\searrow$ be the order on the boxes of $c\times d$ where
\[x\searrow y \mbox{ \ if $x$ is weakly northwest of $y$.}\]
Note that if $x\searrow y$,
then in any increasing filling of $c \times d$, the label of $x$
must be weakly less than that of $y$.


\begin{Claim}\label{mainprop}
Let $\lambda = c\times d$, and let $T\in {\tt INC}(\lambda)$.
Consider a sequence of reverse $K{\tt jdt}$ slides applied to $T$, and
its corresponding {\tt switch} sequence (\ref{eqn:firsttabseq}). Then for each
$i$ we have that:
\begin{enumerate}
\item[(I)] The {\tt switch} from $T_{i-1}$ to $T_i$ is uniform, and hence
boxes of origin are well-defined for~$T_i$.
\item[(II)] Given
entries $\alpha,\beta$ of $T_i$, if $\bor_i(\beta)$ is
to the right of $\bor_i(\alpha)$ and in the same row, then $\beta$ is
strictly
east and weakly north of $\alpha$. Also, the transposed version holds:
if $\bor_i(\beta)$ is south of $\bor_i(\alpha)$ and in the same column then
$\beta$ is strictly south and weakly west of~$\alpha$.
\item[(III)] For any bullet in $T_i$, if the entries to its immediate
north and west are
$\alpha,\beta$, then $\bor_i(\alpha)$ and $\bor_i(\beta)$
are comparable with respect to $\searrow$.
\end{enumerate}
\end{Claim}

\begin{proof}
We prove the claims by simultaneous induction on $i\geq 0$.

The assertions are all trivial for $i=0$.  Suppose we know them for
$i-1$.

$(I)_i$:  By definition, we need to show that all
the entries of any short ribbon involved in a {\tt switch} have the same
box of origin.  By $(III)_{i-1}$, we know that, for any bullet in
$T_{i-1}$,
either its north and west neighbors have the same box of origin, or their
boxes of origin are different but comparable by $\searrow$.  In the former case we are
satisfied, and in the latter case, their labels must be different, so at most one
of these two boxes will be involved in the {\tt switch} from $T_{i-1}$
to $T_i$.

$(II)_i$: It suffices to verify the assertion in the case that
$\bor_i(\alpha)$ and $\bor_i(\beta)$ are adjacent, since the general result follows by transitivity.  So suppose
that we have two boxes in $c\times d$ with $y$ immediately to the right of
$x$ (the proof of the transposed version is similar). If
\begin{equation}
\label{eqn:settingxandy}
\alpha,\beta \mbox{\ in $T_{i-1}$ with
$\bor_{i-1}(\alpha)=x$, $\bor_{i-1}(\beta)=y$,}
\end{equation}
then we know that $\beta$ strictly east and weakly
north of $\alpha$ by $(II)_{i-1}$.  Suppose that this
property
fails for $T_i$.  Then either:
\begin{itemize}
\item [(i)]$\alpha,\beta$
are in the same row in $T_{i-1}$, and the {\tt switch} moves $\beta$ south,
or
\item[(ii)]$\beta$ is in the column
immediately to the right of $\alpha$ in $T_{i-1}$
and the {\tt switch} moves $\alpha$
east.
\end{itemize}

Consider (i). Suppose $\alpha$ and $\beta$ are not adjacent in
 $T_{i-1}$ and let $\gamma$ be some entry between them in the same
row. Where is $\bor_{i-1}(\gamma)$?
Since the entry in $\gamma$ must be greater
than that in $\alpha$, we know $\bor_{i-1}(\gamma)\not \searrow x$.
Similarly,
$y \not\searrow \bor_{i-1}(\gamma)$. Hence $\bor_{i-1}(\gamma)$ is strictly north and strictly east of $x$, or
$\bor_{i-1}(\gamma)$ is strictly south and strictly west of $y$.

In the former case, let $w$ be the box in $c\times d$ with $w$ in the same row
as $x$ and in the same column as $\bor_{i-1}(\gamma)$.  Let $\delta$ be some
box in $T_{i-1}$ with $\bor_{i-1}(\delta)=w$.  Applying $(II)_{i-1}$ to
$\alpha,\delta$ and to $\gamma,\delta$, we find that $\gamma$ is
strictly north and strictly east of $\alpha$, but this contradicts our
assumption that they lie in the same row.  A similar argument shows
that it is impossible for $\bor_{i-1}(\gamma)$ to lie strictly south and
strictly west of $y$.  All these statements taken together show that there is
no possible position for $\bor_{i-1}(\gamma)$.  It follows that $\alpha$ and
$\beta$ must be adjacent.
Now the configuration must be $\cdots \tableau{{\alpha}&{\beta}\\{\theta}&{\bullet}}\cdots$.
Since we are {\tt switch}ing $\beta$ and $\bullet$ and $\theta$ is to the
left of $\bullet$, we must have $\theta\leq \beta$.
This, combined with $(III)_{i-1}$ implies that
$\bor_{i-1}(\theta)\searrow \bor_{i-1}(\beta)$. By
$(II)_{i-1}$ applied
to $\theta$ and $\alpha$, and in view of (\ref{eqn:settingxandy}),
we must have
$\bor_{i-1}(\theta)\ne \bor_{i-1}(\beta)$.
Since $\theta>\alpha$,
we must also have
$\bor_{i-1}(\theta)\not\searrow \bor_{i-1}(\alpha)$.
Together these three observations imply that
$\bor_{i-1}(\theta)$ must be in the column strictly above $\bor_{i-1}(\beta)$, but this violates
the transposed version of
$(II)_{i-1}$ for $\beta,\theta$.

Consider (ii).  This is impossible in $K{\tt jdt}$, because (\ref{eqn:settingxandy}) implies
$\beta>\alpha$,
and so the entry immediately above the bullet is greater than $\alpha$,
which contradicts the fact that we are about to {\tt switch}~$\alpha$.

$(III)_i$: Suppose otherwise.  So we have the following configuration:
$\cdots \tableau{&{\alpha}\\{\beta}&{\bullet}}\cdots$
and $\bor_i(\alpha)$ and $\bor_i(\beta)$ are incomparable under $\searrow$.
Suppose first that $\bor_i(\alpha)$ is strictly south and strictly west of
$\bor_i(\beta)$.
Let $\epsilon$ be an entry of $T_{i}$ such that $\bor_i(\epsilon)$
lies in the same row
as $\bor_i(\alpha)$ and the same column as $\bor_i(\beta)$.  Applying
$(II)_i$ to $\alpha,\epsilon$ and to $\beta,\epsilon$, we find that
$\alpha$ must be weakly south and west of $\beta$, which contradicts
our assumption.

Now suppose that $\bor_i(\alpha)$ is strictly north and strictly east of
$\bor_i(\beta)$.
Let $\epsilon$ be a box of $T_i$
such that $\bor_i(\epsilon)$ lies in
the same row as $\bor_i(\beta)$ and the same column as $\bor_i(\alpha)$.  By
$(II)_i$, $\epsilon$ must lie strictly east of $\beta$, weakly west
of $\alpha$, strictly
south of $\alpha$, and weakly north of $\beta$ so $\epsilon$
must lie in the box where in fact there is a bullet.  This is a
contradiction.
\end{proof}

\begin{Claim}\label{claim:theclaim}
Let $A$ and $B$ be two tableaux of shape $c\times d$.  Consider a common
sequence of reverse $K${\tt jdt} slides applied to each of them, and
consider the resulting {\tt switch} sequences (\ref{eqn:Aseq}) and (\ref{eqn:Bseq}).
\begin{itemize}
\item[(A)] the configurations of $A_i$ and $B_i$ are the same;
\item[(B)] the {\tt switch} from $A_{i-1}$ to $A_{i}$ and the {\tt switch} from $B_{i-1}$ to
$B_{i}$ are both uniform; and
\item[(C)] for each filled box $x$ of the common shape of $A_i$, $B_i$, the box
of origin ${\mathfrak B}_i(x)$ is the same in $A_i$ and $B_i$.
\end{itemize}
\end{Claim}
\begin{proof}
The assertions (A), (B) and (C) are trivially true if $i=0$.  So suppose that they are true
up to step $i-1$.  It follows from Claim \ref{mainprop} (III) that, for
each bullet which will move on the $i$-th step, it will move the same
way in both $A_{i-1}\mapsto A_i$ and $B_{i-1}\mapsto B_i$,
or in other words, the configurations of $A_i$ and $B_i$ are the same. Thus (A) holds. From Claim \ref{mainprop} (I), we
know that the {\tt switch}es from $A_{i-1}$ to $A_i$ and from $B_{i-1}$ to $B_i$
are uniform, establishing (B).  Since the box of origin of any filled box $x$ in
$A_{i-1}$ and $B_{i-1}$ were the same, and the same boxes moved on the
$i$-th step in the two tableaux, it follows that each filled box $x$
in $A_i$ and $B_i$ has the same box of origin, establishing (C). Hence Claim~\ref{claim:theclaim} holds, and
thus the special case of
Theorem~\ref{thm:dualequiv}, where we apply only reverse slides, now holds,
since that is part of the induction claims which we have just established.
\end{proof}

The general case where we use arbitrary slides for (\ref{eqn:Aseq}) and
(\ref{eqn:Bseq}) reduces to the above special case, as in \cite{Haiman}.
The slides in (\ref{eqn:Aseq}) and (\ref{eqn:Bseq}) begin with reverse
$K{\tt jdt}$ slides and then (possibly) some forward $K{\tt jdt}$ slides:
\begin{equation}
\label{eqn:directionseq}
{\tt rev}K{\tt jdt}_{\{\cdots\}}, {\tt rev}K{\tt jdt}_{\{\cdots\}},\cdots,
 {\tt rev}K{\tt jdt}_{\{\cdots\}},K{\tt jdt}_{\{\cdots\}},\cdots
\end{equation}
Call the appearance of a ``${\tt rev}K{\tt jdt}_{\{\cdots\}}$'' followed by
a ``$K{\tt jdt}_{\{\cdots\}}$'', or vice versa, a {\bf direction change}.

To conclude
our proof of the theorem, we induct on the number of direction changes.
The base case where there are no direction changes,
we have checked above. For our induction step, we consider the moment in
(\ref{eqn:directionseq}) where the first direction change occurs.  Suppose
that we have applied ${\tt rev}K{\tt jdt}$ slides to get to $A_j$ and
$B_j$, and that $A_{j+1}$ and $B_{j+1}$ are then each obtained by a $K{\tt jdt}$
slide.
We find a sequence of slides which reverse rectifies $A_j$ and $B_j$ to the southeast corner
of our ambient rectangle $\Lambda$.

We need to know the common shape of these
reverse rectifications is the ``expected'' one (this is not needed in the
corresponding argument in \cite{Haiman} since there he shows standard Young tableaux of
any
straight shape are dual equivalent):

\begin{Lemma}
\label{lemma:recttorect}
If $T\in {\tt INC}(c\times d)$ then the reverse rectification ${\tt rev}K{\tt rect}(T)$
inside
$\Lambda$ is a tableau of shape $c\times d$ placed at the
southeast corner of $\Lambda$.
\end{Lemma}
\noindent
\begin{proof}
This follows from
Claim~\ref{mainprop} (II). The labels of
${\tt rev}K{\tt rect}(T)$
whose boxes of origin are in row
$c$ of $T$ must lie (right justified) in the bottom row of $\Lambda$, since
they form a horizontal strip. Similarly, the labels of
${\tt rev}K{\tt rect}(T)$ whose boxes of origin are in row
$c-1$ must lie right justified and layer on top of the row $c$ labels, for
the same reason, etc.
\end{proof}

Let ${\widehat A}:={\tt rev}K{\tt rect}(A) \mbox{\ and \ } {\widehat B}:=
{\tt rev}K{\tt rect}(B)$,
the shape of both of which we now know to be a $c\times d$ rectangle in
the southeast corner of $\Lambda$,  by Lemma~\ref{lemma:recttorect}.
Since $K{\tt jdt}$ slides are reversible, we can consider
the $K{\tt jdt}$ steps starting from $\widehat A$ and $\widehat B$
that undo the sequence of
${\tt rev}K{\tt jdt}$ steps taking $A_j$ to ${\widehat A}$ and
$B_j$ to ${\widehat B}$. This returns us to $A_j$ and $B_j$
respectively,
and we can continue with the remainder of the sequence (\ref{eqn:directionseq}).
However our new sequence, starting at ${\widehat A}$ and ${\widehat B}$,
involves one fewer direction change than our original sequence. Hence we are done by induction.

This completes the proof of Theorem~\ref{thm:dualequiv}.
\end{proof}

We remark that
for our application to the proof of Theorem~\ref{thm:main},
we actually only need
that rectangular tableaux of the same shape are $K$-dual
equivalent under reverse $K{\tt jdt}$ slides.

\subsection{Conclusion of the proofs}
Let $S_{\lambda}$ be the
superstandard tableau of shape $\lambda$.
Our proofs utilize the following theorem \cite[Theorem~1.2]{Thomas.Yong:V}:

\begin{Theorem}
\label{thm:welldefined}
Let $T\in {\tt INC}(\nu/\mu)$. If $K{\tt rect}(T)$ is
a superstandard tableau $S_{\lambda}$ for some rectification order, then
$K{\tt rect}(T)=S_{\lambda}$ for any rectification order.
\end{Theorem}

It is Theorem~\ref{thm:welldefined} that makes Theorem~\ref{thm:TY5main} well-defined.
Moreover, Theorem~\ref{thm:welldefined} has the following reformulation, not 
explicitly
mentioned in \cite{Thomas.Yong:V}:

\begin{Corollary}\label{cor:comb}
Fix $C\in {\tt INC}(\mu)$. The
number of tableaux $T$ of shape $\nu/\lambda$ that rectify to $C$ with
respect to the rectification order given by $S_\lambda$,
does not depend on the choice of $C$.
\end{Corollary}

\begin{proof}
We refer the reader to \cite[Section~3]{Thomas.Yong:V}
for the definition of
\[{\tt Kinfusion}(T,U)=({\tt Kinfusion}_1(T,U),{\tt Kinfusion}_2(T,U)).\]

$V$ rectifies with respect to the rectification order
$S_\lambda$ to $C$, if and only if
${\tt Kinfusion}(S_\lambda,V)=(C,W)$ for some $W$ of shape $\nu/\mu$.
Since {\tt Kinfusion} is an involution \cite[Theorem 3.1]{Thomas.Yong:V},
it defines a bijection between
the following two sets \begin{enumerate}\item
$\{V$ of shape $\nu/\lambda$ that rectifies with respect to
$S_\lambda$ to $C\}$
\item $\{W$ of shape $\nu/\mu$ that rectifies with respect to $C$ to
$S_\lambda\}$.
\end{enumerate}
By Theorem~\ref{thm:welldefined}, the number of $W$ as in (2) is independent
of the choice of $C$, so the same must be true of the set (1) also.
\end{proof}

\noindent\emph{Proof of Theorem~\ref{thm:main}:}
A choice of $K$-rectification order amounts to a choice of $U\in {\tt INC}(R)$.
Consider ${\tt Kinfusion}(U,T)=(C,D)$, where by definition,
$C={\tt Krect}(T)$. Let
$U_0:=U\mapsto U_1\mapsto U_2\mapsto \cdots \mapsto U_m$,
be a sequence of increasing tableaux where $U_i$ is the result of
reverse $K{\tt jdt}$ slides into corners marked by the labels $i$ of
$T$.  The boxes labeled $i$ in $C$ are exactly those
vacated
as we move from $U_{i-1}$ to $U_i$. Thus, $C$ is completely determined by
${\tt shape}(U_i) \mbox{\ for $i=0,\ldots,m$.}$
This sequence of shapes
is independent of the choice of $U$, by Theorem~\ref{thm:dualequiv},
hence (I) holds.

To prove (II), observe that, by (I), we may
assume that our rectification order is given by $S_R$. (II) now follows from
Corollary~\ref{cor:comb}.
\qed

\medskip
\noindent
\emph{Proof of Theorem~\ref{cor:main}:}
In Section~2, we proved by a geometrical argument that $D_{\lambda,\mu}^{\nu}=C_{\Omega^{\vee},\nu}^{\lambda\dagger\mu}$.
In view of this, the theorem follows from
Theorem~\ref{thm:main} combined with Theorem~\ref{thm:TY5main}.
\qed

\section{Examples}

\begin{Example} Let $\lambda=\tableau{{\ }&{ \ }}$ and
$\mu=\tableau{{\ }&{\ }\\{\ }}$. There are fifteen increasing tableaux $T$
of shape $\lambda\star\mu$ with labels from $\{1,2,3\}$.
We list these below, together with
$K{\tt rect}(T)$, collected into $7$ groups, by the shape of $K{\tt
rect}(T)$:
\[\left\{\tableau{&&{1}&{2}\\{1}&{2}\\{3}}\mapsto
\tableau{{1}&{2}\\{2}\\{3}},\tableau{&&{1}&{3}\\{2}&{3}\\{3}}\mapsto
\tableau{{1}&{3}\\{2}\\{3}}\right\}_{\smtableau{{ \ }&{\ }\\{\ }\\{\ }}},
\left\{\tableau{&&{1}&{2}\\{1}&{3}\\{2}}\mapsto \tableau{{1}&{2}\\{2}&{3}}
\right\}_{\smtableau{{ \ }&{\ }\\{\ }&{\ }}},
\]
\[\left\{\tableau{&&{1}&{2}\\{1}&{3}\\{3}}\mapsto
\tableau{{1}&{2}\\{3}},
\tableau{&&{1}&{3}\\{1}&{3}\\{2}}\mapsto
\tableau{{1}&{3}\\{2}},
\tableau{&&{1}&{2}\\{1}&{2}\\{2}}\mapsto
\tableau{{1}&{2}\\{2}},\right.\]
\[\left.
\tableau{&&{1}&{3}\\{1}&{3}\\{3}}\mapsto
\tableau{{1}&{3}\\{3}},
\tableau{&&{2}&{3}\\{2}&{3}\\{3}}\mapsto
\tableau{{2}&{3}\\{3}}
\right\}_{\smtableau{{\ }&{\ }\\{\ }}},\]
\[\left\{\tableau{&&{1}&{3}\\{1}&{2}\\{3}}\mapsto
\tableau{{1}&{2}&{3}\\{2}\\{3}}
\right\}_{\smtableau{{\ }&{\ }&{\ }\\{\ }\\{ \ }}},
\left\{\tableau{&&{1}&{2}\\{2}&{3}\\{3}}\mapsto
\tableau{{1}&{2}\\{2}&{3}\\{3}}
\right\}_{\smtableau{{\ }&{\ }\\{\ }&{\ }\\{\ }}},\]
\[\left\{\tableau{&&{2}&{3}\\{1}&{3}\\{2}}\mapsto
\tableau{{1}&{2}&{3}\\{2}&{3}}
\right\}_{\smtableau{{\ }&{\ }&{\ }\\{\ }&{\ }}},
\]
\[\left\{\tableau{&&{2}&{3}\\{1}&{2}\\{3}} \ \&  \
\tableau{&&{2}&{3}\\{1}&{3}\\{3}}
\mapsto
\tableau{{1}&{2}&{3}\\{3}},
\tableau{&&{2}&{3}\\{1}&{2}\\{2}} \ \&  \
\tableau{&&{1}&{3}\\{1}&{2}
\\{2}}
\mapsto
\tableau{{1}&{2}&{3}\\{2}}
\right\}_{\smtableau{{\ }&{\ }&{\ }\\{\ }}}.
\]

Note that for each fixed straight shape $\nu$, each increasing tableau of shape $\nu$
appears the same number of times, in agreement with Theorem~\ref{thm:main}(II).
In particular, the last group witnesses $D_{(2),(2,1)}^{(3,1)}=-2$.\qed
\end{Example}

A.~Buch's rule \cite[Section~6]{Buch:KLR} states
that if
$\lambda=(\lambda_1,\ldots,\lambda_p)  \mbox{\ and $\mu=(\mu_1,\ldots,\mu_q)$}$
then $D_{\lambda,\mu}^{\nu}$ equals $(-1)^{|\nu|+|\lambda|+|\mu|}$ times the number of
set-valued tableaux $Y$ of shape $\nu$ such that the bottom-up and left to
right
reading word is a partial reverse lattice word with respect to both of the
intervals $[1,p]$ and $[p+1,p+q]$, and has content
$(\lambda_1,\ldots,\lambda_p,\mu_1,\ldots,\mu_q)$.
For brevity, we refer
the
interested reader to
his paper for precise definitions. The following
two
tableaux of shape $(3,1)$ satisfy Buch's rule (for $p=1$ and $q=2$):
\[\ktableau{{1}&{1,2}&{2}\\{3}}, \ \ktableau{{1}&{1}&{2}\\{2,3}} \mbox{\ \ \
with respective reading words: $31122$ and $23112$.}\]
The reader can check that these are the only
set-valued tableaux of this shape and content $(2,2,1)$ such that
there are always more $2$'s than $3$'s as one reads these reading words
from
right to left. This therefore agrees with our independent computation.

The following example illustrates Theorem~\ref{thm:main3}:
\begin{Example}
We have that $E_{(1),(1)}^{(2,1)}=-3$ since the coefficient is equal to $(-1)^{3-1-1}$ times the following witnessing tableaux
from ${\widehat{{\tt INC}}}((2,1)/(1))$:
$\tableau{{\ }&{1}\\{1}}, \ \ \tableau{{\ }&{1}\\{X}}, \ \ \tableau{{\ }&{X}\\{1}}$.
\qed
\end{Example}

\section{Theorem~\ref{thm:main} is sharp}
If we replace the rectangle $R$ in Theorem~\ref{thm:main} by any nonrectangular
shape $\lambda$, one can always construct a shape $\nu\supseteq\lambda$
and a tableau $T\in {\tt INC}(\nu/\lambda)$ such that the conclusion (I)
of the theorem fails and hence (II) no longer makes sense. Thus Theorem~\ref{thm:main} is sharp.

To see this, first compare the rectification of the tableau
\begin{equation}
\label{eqn:counterex}
T=\tableau{&&{2}\\&{1}&{4}\\{1}&{3}} \mbox{\ via \ }
K{\tt infusion}_1\left(\tableau{{1}&{2}\\{3}},T\right) \mbox{ \ and \ }
K{\tt infusion}_1\left(\tableau{{1}&{3}\\{2}},T\right),
\end{equation}
where as usual we use a labeling of $\lambda=(2,1)$ to indicate the order
of $K{\tt rectification}$, see \cite[Section~3]{Thomas.Yong:V}. The reader can check that these two $K$-infusions
give, respectively,
$\tableau{{1}&{2}&{4}\\{3}} \mbox{\ \ and \ \ }
\tableau{{1}&{2}&{4}\\{3}&{4}}$.
This example was first mentioned in
\cite[Section~8]{Thomas.Yong:V}.

Now, suppose $\lambda$ is not a rectangle. One can
generalize the above counterexample. For brevity, we do this for a
particular
instance of $\lambda$ and leave the straightforward details of the general argument to the interested
reader.

\begin{Example}
Let $\lambda=(6,6,3,1)$ and let $\nu=(7,6,5,2,1)$ with the indicated
filling
${\widetilde T}\in {\tt INC}(\nu/\lambda)$, that appropriately ``spreads out'' the $T$ from (\ref{eqn:counterex}):
\[{\widetilde T}=\tableau{{\ }&{\ }&{\ }&{ \ }&{ \ }&{ \ }&{2 }\\{\ }&{\
}&{\ }&{ \ }&{\ }&{\ }\\
{\ }&{\ }&{\ }&{1}&{4}\\{\ }&{3}\\{1}}.\]
One can choose a partial $K$-rectification of the
numerical labels to obtain that $T$. We illustrate
this for our example below, but this clearly extends to any $\lambda$:
\[{\widetilde T}=\tableau{{\ }&{\ }&{\ }&{ \ }&{ \ }&{ \ }&{2 }\\{\ }&{\
}&{\ }&{ \ }&{\ }&{\ }\\
{\ }&{\ }&{\ }&{1}&{4}\\{1 }&{3}\\{\ }}\mapsto
\tableau{{\ }&{\ }&{\ }&{ \ }&{ \ }&{ \ }&{2 }\\{\ }&{\ }&{\ }&{ \ }&{\ }&{\
}\\
{\ }&{1 }&{4 }&{\ }&{\ }\\{1 }&{3}\\{\ }}
\mapsto
\tableau{{\ }&{\ }&{\ }&{ \ }&{ \ }&{ \ }&{2 }\\{\ }&{1 }&{4 }&{ \ }&{\ }&{\
}\\
{1 }&{3 }&{\ }&{\ }&{\ }\\{\ }&{\ }\\{\ }}\mapsto T=\tableau{{\ }&{\
}&{2}\\{\ }&{1}&{4}\\{1}&{3}}
\]
Completing the $K$-rectification in two ways as in
(\ref{eqn:counterex}) gives different $K$-rectifications of ${\widetilde T}$.\qed
\end{Example}

Note that one can always use $\nu$ that has at most one more row and one more
column than $\lambda$ does.

\section{Non-associative products of tableaux}
Theorem~\ref{thm:main}(I) allows us to define a product $\odot$
on increasing tableaux $T,U$ by setting
\[T\odot U=K{\tt rect}(T\star U)\]
where $T\star U$ means that $T$ and $U$ are arranged
corner to corner, so that ${\tt shape}(T\star U)={\tt shape}(T)\star{\tt shape}(U)$.
However, the fact that $K$-rectification is not well-defined for general skew shapes,
means that this
product, unlike the plactic product, is not associative.

In \cite{BKSTY}, another non-associative product $\diamond$ on tableaux was constructed. To define it
we briefly recall the existence of {\bf Hecke insertion}, that
inserts a non-negative integer $x$ into an increasing tableau $Z$, resulting in another increasing tableau $Y$.
We denote this by $Z\leftarrow x$.
This insertion is a generalization of the insertion algorithms of C.~Schensted \cite{Schensted} and of
P.~Edelman--C.~Greene \cite{Edelman.Greene}. Hecke insertion was reformulated in
\cite{Thomas.Yong:VI} in terms of $K{\tt rect}$. In the terminology of this paper, we have:
\begin{equation}
\label{eqn:heckeiskjdt}
Z\leftarrow x = Z\odot \tableau{{x}}.
\end{equation}
More generally we can speak of inserting an increasing tableau $W$ into
$Z$. If $(w_1,\ldots,w_p)$ is the left to right, bottom to top row reading word of $W$, then define the product
\[Z\diamond W=(((Z\leftarrow w_1)\leftarrow w_2)\leftarrow\cdots\leftarrow w_p).\]
In view of (\ref{eqn:heckeiskjdt}) it is natural to suspect that the $\odot$ and $\diamond$ products are the same. However
this is not true, as one can verify that
\[\tableau{{1}&{2}&{3}\\{2}&{4}&{5}}\odot \tableau{{1}&{2}}\neq
\tableau{{1}&{2}&{3}\\{2}&{4}&{5}}\diamond \tableau{{1}&{2}}\left(=
\left(\tableau{{1}&{2}&{3}\\{2}&{4}&{5}}\odot \tableau{{1}}\right)\odot\tableau{{2}}\right),
\]
since the left and right hand sides of the ``$\neq$''
are respectively
$\tableau{{1}&{2}&{3}\\{2}&{3}&{5}\\{4}&{5}} \mbox{\ and \ }
\tableau{{1}&{2}&{3}\\{2}&{3}&{5}\\{4}}$.
Note that underlying the example of the inequivalence of $\odot$ and $\diamond$ is
the non-associativity of the
$\odot$ product, i.e.,
\[\tableau{{1}&{2}&{3}\\{2}&{4}&{5}}\odot \left(\tableau{{1}}\odot\tableau{{2}}\right)\neq
\left(\tableau{{1}&{2}&{3}\\{2}&{4}&{5}}\odot \tableau{{1}}\right)\odot\tableau{{2}}.\]



\section*{Acknowledgments}
The authors thank Anders Buch, Thomas Lam, Aaron Lauve, Li Li, Kevin Purbhoo,
Pavlo Pylyavskyy and David Speyer
for helpful conversations. We especially thank Allen Knutson,
for his questions that led to the inclusion of Section~2.
HT was supported by an NSERC Discovery Grant. AY was supported by NSF
grants
DMS 0601010 and DMS 0901331.

\end{document}